\documentclass{amsart}
\usepackage{mathabx,mathtools,amsrefs,mathrsfs,math,enumerate,multirow}
\usepackage{tikz}
\usetikzlibrary{arrows,automata,calc}
\numberwithin{equation}{section}

\newenvironment{fsa}[1][auto]{\begin{tikzpicture}[->,>=stealth',
    shorten >=1pt,auto,node distance=3cm,double distance between line centers=0.45ex,
    initial text=,accepting/.style=accepting by arrow,
    every state/.style={inner sep=3pt,minimum size=0pt,fill=gray,text=white,draw=none},
    every loop/.style={looseness=12},semithick,#1]}{\end{tikzpicture}}

\newenvironment{dualmoore}[1][auto]{\begin{tikzpicture}[->,>=stealth',
    shorten >=1pt,auto,node distance=2cm,double distance between line centers=0.45ex,
    initial text=,accepting/.style=accepting by arrow,
    every state/.style={rectangle,inner sep=3pt,minimum size=0pt},
    every loop/.style={looseness=12},semithick,#1]}{\end{tikzpicture}}

\makeatletter
\newcommand*{\rom}[1]{\expandafter\@slowromancap\romannumeral #1@}
\makeatother

\newtheorem{mainthm}{Theorem}
\renewcommand{\themainthm}{\Alph{mainthm}}

\newcommand\mm{{\mathbf M}}

\begin{document}
\title{The word and order problems for self-similar and automata groups}
\author{Laurent Bartholdi}
\author{Ivan Mitrofanov}
\date{October 25, 2017}
\address{D\'epartement de Math\'ematiques et Applications, \'Ecole Normale Sup\'erieure, Paris \textit{and} Mathematisches Institut, Georg-August Universit\"at zu G\"ottingen}
\email{laurent.bartholdi@gmail.com}
\email{phortim@yandex.ru}

\thanks{This work is supported by the ``@raction'' grant ANR-14-ACHN-0018-01}
\begin{abstract}
  We prove that the word problem is undecidable in functionally
  recursive groups, and that the order problem is undecidable in
  automata groups, even under the assumption that they are
  contracting.
\end{abstract}
\maketitle

\section{Introduction}
Let $A$ be a finite set, and consider a group $G$ acting faithfully
and ``self-similarly'' on the set $A^*$ of words over $A$. This means
that every $g\in G$ acts in the form
\begin{equation}\label{eq:action}
  (a_1\dots a_n)^g=a'_1(a_2\dots a_n)^{g'}
\end{equation}
for some $a'_1\in A$ and some $g'\in G$ depending only on $a_1,g$; we
encode them as $(g',a'_1)=\overline\Phi(a_1,g)$ for a map
$\overline\Phi\colon A\times G\to G\times A$. If furthermore $G$ is
finitely generated (say by a finite set $S$, so $G$ is a quotient
$F_S\twoheadrightarrow G$ of the free group on $S$), then its action
may be described by finite data, namely a lift
$\Phi\colon A\times S\to F_S\times A$ of the restriction of
$\overline\Phi$ to the generators of $G$. A finitely generated group
given in this manner is called \emph{functionally
  recursive}~\cite{brunner-s:auto}*{\S3}, or \emph{self-similar}; we
call $G$ the group \emph{presented} by $\Phi$, and write
$G=\langle\Phi\rangle$. We call $\Phi$ an (asynchronous)
\emph{transducer}.

Large classes of finitely generated groups can be presented as
functionally recursive ones; notably, all the ``iterated monodromy
groups'' of Nekrashevych~\cite{nekrashevych:ssg}, and the automata
groups mentioned in~\S\ref{ss:automata} below.

Even though the map $\Phi$ completely determines the action of $G$,
and therefore $G$ itself, it is unclear how much of $G$ is known from
$\Phi$. Our first result is as negative as can be:
\begin{mainthm}\label{thm:wp}
  There is no algorithm that, given
  $\Phi\colon A\times S\to F_S\times A$ and $s\in S$, determines
  whether $s=1$ in $\langle\Phi\rangle$.
\end{mainthm}

\subsection{Automata groups}\label{ss:automata}
Assume now that $G$ is a functionally recursive group, and that in the
action~\eqref{eq:action} the elements $g'$ have at most the length of
$g$, in the generating set $S$. Then, up to replacing $S$ by
$S\cup S^{-1}\cup\{1\}$, the map $\Phi$ takes the form
$\Phi\colon A\times S\to S\times A$; we call it a \emph{finite state
  transducer}. The group $G$ is called an \emph{automata group}; these
form a notorious class of groups, containing all finitely generated
linear groups as well as infinite torsion groups such as the
``Grigorchuk group''~\cite{grigorchuk:burnside} and ``Gupta-Sidki
groups''~\cite{gupta-s:burnside}. The Grigorchuk group is also a group
of intermediate word-growth, and was used to settle the Milnor problem
on group growth~\cite{grigorchuk:growth}.

The action of $S$, and of $G$ itself, may be conveniently described by
a finite labeled graph called its \emph{Moore diagram}. Consider the
directed graph $\Gamma$ with vertex set $S$ and an edge from $s$ to
$t$ labeled $(a,b)$ whenever $\Phi(a,s)=(t,b)$; then the action of
$s\in S$ on $A^*$ is determined as follows: given
$a_1\dots a_n\in A^*$, find the unique path in $\Gamma$ starting at
$s$ and whose first label letters read $a_1\dots a_n$; let
$b_1\dots b_n$ be the second label letters; then
$(a_1\dots a_n)^s=b_1\dots b_n$. See Figure~\ref{fig:grigorchuk} for
the graph $\Gamma$ describing the Grigorchuk group.

Every element of $G$ (say represented by a word $w$ of length $n$ in
$S$) admits a similar description, but now using a graph with vertex
set $S^n$. The word $w$ represents the identity in $G$ if and only if
at every vertex reachable from $w$ all the outgoing edges have labels
in $\{(a,a)\mid a\in A\}$. It follows that the word problem is
decidable in $G$, and even belongs to \textsc{LinSpace} (and therefore
to \textsc{ExpTime}); but that is about as much as is known. We
consider the ``order problem'' (determine the order of an element),
which was raised at the end of last century by Sidki and by
Grigorchuk, Nekrashevych and
Sushchansky~\cite{grigorchuk-n-s:automata}*{Problem 7.2.1(a)}, to
which Gillibert announced a solution in July 2017; his proof appears
in~\cite{gillibert:gporderpb}:
\begin{mainthm}\label{thm:torsion}
  There is no algorithm that, given
  $\Phi\colon A\times S\to S\times A$ and $s\in S$, determines the
  order of $s$ in $\langle\Phi\rangle$, namely the cardinality of
  $\langle s\rangle$.
\end{mainthm}

Worse than that, the action is uncomputable in the following sense:
consider the natural extension of the action of $\langle\Phi\rangle$
to $A^\infty$. Then we have the following variants of
Theorems~\ref{thm:wp} and~\ref{thm:torsion}:
\setcounter{mainthm}{0}\renewcommand\themainthm{\Alph{mainthm}$'$}
\begin{mainthm}\label{thm:wp_orbit}
  There is no algorithm that, given
  $\Phi\colon A\times S\to F_S\times A$ and $a\in A$ and $s\in S$,
  determines whether $a^\infty$ is fixed by $s$.
\end{mainthm}

\begin{mainthm}\label{thm:torsion_orbit}
  There is no algorithm that, given
  $\Phi\colon A\times S\to S\times A$ and $a\in A$ and $s\in S$,
  determines the cardinality of the orbit of $a^\infty$ under
  $\langle s\rangle$.
\end{mainthm}

\noindent Finally, the results in Theorems~\ref{thm:wp}
and~\ref{thm:torsion} can be made uniform as follows:
\setcounter{mainthm}{0}\renewcommand\themainthm{\Alph{mainthm}$''$}
\begin{mainthm}\label{thm:wp_uniform}
  There is a functionally recursive group $\langle\Phi\rangle$ with
  $\Phi\colon A\times S\to F_S\times A$ such that
  $\{s\in F_S\mid s=1\text{ in }\langle\Phi\rangle\}$ is not
  recursive.
\end{mainthm}

\begin{mainthm}\label{thm:torsion_uniform}
  There is an automata group $\langle\Phi\rangle$ with
  $\Phi\colon A\times S\to S\times A$, and two states $s,t\in S$, such
  that the set $\{n\in\N\mid s t^n\text{ has finite order}\}$ is not
  recursive.
\end{mainthm}

\begin{figure}
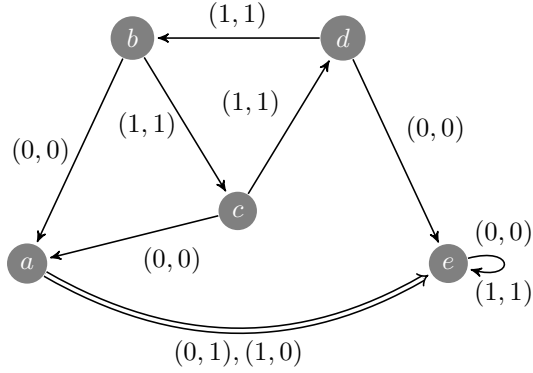

\begin{center}
\begin{fsa}[scale=1]
  \node[state] (b) at (1.4,3) {$b$};
  \node[state] (d) at (4.2,3) {$d$};
  \node[state] (c) at (2.8,0.7) {$c$};
  \node[state] (a) at (0,0) {$a$};
  \node[state] (e) at (5.6,0) {$e$};
  \path (b) edge node[left] {$(1,1)$} (c) edge node[left] {$(0,0)$} (a)
        (c) edge node {$(1,1)$} (d) edge node {$(0,0)$} (a)
        (d) edge node[above] {$(1,1)$} (b) edge node {$(0,0)$} (e)
        (a) edge[-implies,double,bend right=30] node[below] {$(0,1),(1,0)$} (e)
        (e) edge[loop right] node[above=1mm] {$(0,0)$} node[below=1mm] {$(1,1)$} (e);
\end{fsa}
\end{center}
\caption{The transducer generating the Grigorchuk group. 
Here $A=\{0,1\}$ and $S=\{a,b,c,d,e\}$.}\label{fig:grigorchuk}
\end{figure}

\subsection{Contracting groups}
Assume now that $G$ is a functionally recursive group, and that in the
action~\eqref{eq:action} the elements $g'$ are \emph{shorter} than
$g$, in the generating set $S$, in the sense that there are constants
$\lambda<1$ and $C$ with $|g'|\le\lambda|g|+C$ for all $g\in G$. Then,
up to replacing $S$ by the set of all words of length
$\le C/(1-\lambda)$, we also have $|g'|\le|g|$; we have thus defined a
subclass of automata groups, called \emph{contracting automata groups}
(see~\S\ref{ss:contract} for a more precise definition). Their word
problem is decidable in \textsc{LogSpace} (and therefore in
\textsc{PolyTime}).  We will see, however, that the order and orbit
order problems remain unsolvable in that restricted class:
\renewcommand\themainthm{\Alph{mainthm}}
\begin{mainthm}[= Theorem~\ref{thm:nuclear}]\label{thm:contracting}
  The transducers constructed in Theorems~\ref{thm:torsion}
  and~\ref{thm:torsion_orbit} may be assumed to generate contracting
  groups.
\end{mainthm}

\subsection{Sketch of proofs}
We encode Minsky machines in functionally recursive groups. Minsky
machines (see~\cite{minsky:post}) are restricted Turing machines with
two tapes, which may move the tapes and sense the tape's end but may
not write on them; equivalently, they are finite state automata
equipped with two counters with values in $\N$ that may be
incremented, decremented and tested for $0$.

In all cases, we encode the machine, in state $s$ with counter values
$(m,n)$, by the word $s x^{2^m} y^{2^n}$ in a functionally recursive
group containing elements $x,y$ and an element $s$ for each state of
the machine. The action of the group is so devised that if the machine
evolves to state $(s',m',n')$ then the recursive action is given by
$s' x^{2^{m'}} y^{2^{n'}}$. The image of a prescribed ray under
$s x^{2^m} y^{2^n}$ records the computational steps of the machine
when started in $(s,m,n)$, and in particular whether the machine
reached a final state. We construct an auxiliary element $t$ that only
acts on sequences containing a trace of this final state, and then
$(s x^{2^m} y^{2^n})t(s x^{2^m} y^{2^n})^{-1}$ fixes the original ray
if and only if the machine never reaches the final state. Taking the
commutator of that last element acting only in the neighbourhood of
the original ray yields an expression that is trivial if and only if
the machine never reaches the final state.

Inherently, sometimes the output of the transducer is longer than the
input (e.g., if the machine increments the first counter, the
transducer must replace $x$ by $x^2$). To obtain an automata group, we
have the transducer consume a power of its input word
$s x^{2^m} y^{2^n}$; e.g., the incrementation of the counter may be
performed by erasing every second $s$ and every second block of
$y^{2^n}$'s. The element $s x^{2^m} y^{2^n}$ then may be arranged to
finite order if and only if the machine reaches the final state.

\subsection{Tilings}
Our results on functionally recursive groups and transducers may also be
interpreted in terms of tilings. Let $C$ be a finite set of
\emph{colours}, and $T\subseteq C^{N,E,S,W}$ be a set of \emph{Wang
  tiles}. A \emph{valid tiling} is a map $t\colon\Z^2\to T$ with
$t(x,y)^N=t(x,y+1)^S$ and $t(x,y)^E=t(x+1,y)^W$ for all
$x,y\in\Z^2$. Berger showed in~\cite{berger:undecidability} that it is
undecidable to determine, given $T$, whether there exists a valid
tiling by $T$. This has been improved: for
$\lambda,\mu\in\{N,E,S,W\}$, call a set of tiles
\emph{$\lambda\mu$-deterministic} if for every $c,d\in C$ there exists
at most one tile $u\in T$ with $u^\lambda=c$ and $u^\mu=d$, and
\emph{$\lambda\mu$-complete} if there exists precisely one tile
$u\in T$ with these conditions. Lukkarila showed
in~\cite{lukkarila:4wayundecidable} that the undecidability result
holds even under the restriction that $T$ is
$NE,NW,SE,SW$-deterministic. Clearly a $SW$-complete tileset tiles
uniquely the first quadrant for any choice of colours on the axes.

Our result on the order problem has the following translation into
tilings. We consider tilings of the upper half-plane
$\{(x,y)\mid y\ge0\}$. Then the following problem is undecidable even
for $SE,SW$-complete tilesets: \emph{``given $c\in C$, is there an
  integer $n\in\N$ such that every tiling of the upper half-plane with
  $c^\infty$ on the horizontal axis is horizontally $n$-periodic?''}.

Indeed, given $\Phi\colon A\times S\to S\times A$, set $C=A\sqcup S$
and whenever $\Phi(a,s)=(s',a')$ build a tile with $N,E,S,W$-labels
$s',a',s,a$ respectively; also build tiles with $N,E,S,W$-labels
$c,d,c,d$ for all $(c,d)\in C^2\setminus(S\times A)$. Then the above
tiling problem is satisfied for $c\in S$ if and only if $c$ has finite
order in $\langle\Phi\rangle$.

The word problem may also be translated to a tiling problem, but now
in hyperbolic space. The tileset is now
$T\subseteq C^{N,E,S_1,S_2,W}$. The lattice $\Z^2$ is now
$\Lambda\coloneqq\{2^y(i+x)\mid x,y\in\Z\}\subset\mathbb H$.  A tiling
is a map $t\colon\Lambda\to T$ with $t(2^y(i+x))^E=t(2^y(i+x+1))^W$
and $t(2^y(i+2x))^N=t(2^{y+1}(i+x))^{S_1}$ and
$t(2^y(i+2x+1))^N=t(2^{y+1}(i+x))^{S_2}$ for all $x,y\in\Z$. Tiles are
visualized as pentagons assembling into a tiling of the hyperbolic
plane, invariant under the transformations $z\mapsto z+1$ and
$z\mapsto 2z$:
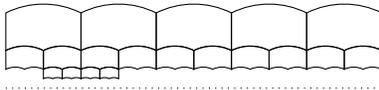
\[\begin{tikzpicture}
    \foreach \x/\s in {0/1,1/1,2/1,3/1,4/1,
      0/0.5,1/0.5,2/0.5,3/0.5,4/0.5,5/0.5,6/0.5,7/0.5,8/0.5,9/0.5,
      2/0.25,3/0.25,4/0.25,5/0.25} do
    \draw[scale=\s] (\x,1) arc (116.5:63.5:1.12) -- +(0,-0.5) arc (63.5:116.5:0.56) arc (63.5:116.5:0.56) -- cycle;
    \draw[dotted] (0,0) -- (5,0);
  \end{tikzpicture}\]

The following problem is undecidable even for $NE,NW$-complete
tilesets: \emph{``given $c\in C$, does every tiling of
$\{x+i y\in\mathbb H\mid x\in[0,1],y\le1\}$with $c$ on the edge from
$i$ to $i+1$ have identical labels on the boundary half-lines
$\{x=0\}$ and $\{x=1\}$?''}.

Indeed by subdividing and inserting the empty state we may assume that
the map $\Phi$ describing our functionally recursive group satisfies
$\Phi(A\times S)\subseteq S^2\times A$; then tiles are defined as
above.

\subsection{History}
Links have been established since the beginning between undecidable
problems in theoretical computer science --- halting of Turing
machines --- and in algebra --- decidability of properties of
algebraic objects. Minsky machines, because of their simplicity, have
been early recognized as useful tools in this correspondence, see
e.g. Gurevich's work~\cite{gurevich:minsky} on identities in
semigroups.

Automata semigroups are defined quite similarly to automata groups;
one merely drops the requirement that the action be by invertible
maps. Decision problems have been extensively studied within the class
of automata
semigroups~\cites{akhavi-klimann-lombardy-mairesse-picantin:finiteness,klimann-mairesse-picantin:computations}. Gillibert
proved in~\cite{gillibert:finiteness} that the order problem is
unsolvable in that class. His proof is based on the undecidability of
Wang's tiling problem~\cite{berger:undecidability}, and harnesses
Kari's solution of the nilpotency problem for cellular
automata~\cite{kari:nilpotency}.

There are usually serious difficulties in converting a solution in
semigroups to one in groups. In particular, the tilings at the heart
of Gillibert's construction give fundamentally non-invertible
transformations of $A^*$.

On the other hand, a direct approach to the order problem succeeded
for the restricted class of ``bounded automata'' groups; Bondarenko,
Sidki and Zapata prove in~\cite{bondarenko-b-s-z:conjugacy} that they
have solvable order problem. Gillibert announced
in~\cite{gillibert:gporderpb} the undecidability of the order problem
in automata groups; his work uses a simulation of arbitrary Turing
machines by transducers via cellular automata.

\section{Functionally recursive groups and Minsky machines}\label{sec:wp}
All our theorems are proven by embedding Minsky machine computations
into functionally recursive groups. Let us recall more precisely the
definition of these machines:

\begin{defn}
  A \emph{Minsky machine} is a computational device $\mm$
  equipped with two integer counters $m,n$ and a finite amount of
  additional memory. It has a finite set $S$ of \emph{states}, an
  \emph{initial state} $s_*\in S$, a \emph{final state}
  $s_\dagger\in S$, and for each state $s\neq s_\dagger$ an
  instruction, which can by any of the following kind:
  \begin{enumerate}[I:]
  \item $(s,m,n)\mapsto(s',m+1,n)$;
  \item $(s,m,n)\mapsto(s',m,n+1)$;
  \item $(s,m,n)\mapsto(s',m+1,n+1)$;
  \item $(s,m,n)\mapsto(s',m-1,n)$, only valid if $m>0$;
  \item $(s,m,n)\mapsto(s',m,n-1)$, only valid if $n>0$;
  \item $(s,m,n)\mapsto(s',n,m)$;
  \item $(s,m,n)\mapsto(m=0\;?\;s':s'',m,n)$;
  \item $(s,m,n)\mapsto(n=0\;?\;s':s'',m,n)$;
  \item $(s,m,n)\mapsto m=0\;?\;(s',m,n):(s'',m-1,n)$;
  \item $(s,m,n)\mapsto n=0\;?\;(s',m,n):(s'',m,n-1)$.
  \end{enumerate}

  \noindent(We use the C style ``?:'' operator, with `$a\;?\;b:c$' meaning `if
  $a$ then $b$ else $c$'.)
  
  As $\mm$ is turned on, its state and counters initialize at
  $(s_0,m_0,n_0)=(s_*,0,0)$, and then $(s_{i+1},m_{i+1},n_{i+1})$ is
  determined from $(s_i,m_i,n_i)$ using the prescribed rules. If at
  some moment $s_i=s_\dagger$ then $\mm$ \emph{stops}; otherwise
  it runs forever.
\end{defn}

We recall the main result on Mealy machines, namely that they are as
powerful as Turing machines:
\begin{prop}[\cite{minsky:post}]\label{prop:minsky}
  (1) There is no algorithm that, given a Minksy machine $\mm$,
  determine whether it stops.

  (2) There is a ``universal'' Minsky machine $\mm$ such that
  \[\{n\in\N\mid \mm\text{ stops when turned on in state
    }(s_*,0,n)\}\]
  is not recursive.\qed
\end{prop}

We also note that only one of the instructions \{I,II\} and III is
necessary, and that in the presence of VI only one of I,II, one of
IV,V, one of VII,VIII and one of IX,X is necessary. Minimal sets of
instructions are \{III,IV,V,VII,VIII\} and \{I,IV,VI,VII\} and
\{III,IX,X\} and \{I,VI,IX\}.

\subsection{Proof of Theorem~\ref{thm:wp_orbit}}
Let $\mm$ be a Minsky machine with stateset $S_0$. Without loss
of generality, we assume that all instructions of $\mm$ are of
type I, VI, IX.

We construct a functionally recursive group $\langle\Phi_\mm\rangle$
presented by $\Phi_\mm\colon A\times S\to F_S\times A$, for sets $A,S$
given as follows: the generating set $S$ consists of
\begin{itemize}
\item elements $x$, $y$, $s_\dagger$, $t$ and $u$;
\item for each state $s_i\in S_0$ of type I or IX, an element $s_i$; 
\item for each state $s_i\in S_0$ of type VI, three elements $s_i, a_i, b_i$.
\end{itemize}

\def\frstate#1#2{{\mathbf{#1}_{#2}}}

The alphabet $A$ consists of:
\begin{itemize}
\item four letters $0$, $1$, $\frstate\dagger1$ and $\frstate\dagger2$;
\item for each state $s_i\in S_0$ of type I, a letter $\frstate i1$; 
\item for each state $s_i\in S_0$ of type IX, two letters $\frstate i1$ and $\frstate i2$;
\item for each state $s_i\in S_0$ of type VI, five letters $\frstate i1,\frstate i2,\dots,\frstate i5$.
\end{itemize}

The map $\Phi_\mm\colon A \times S \to F_S \times A$ is given below,
with $\epsilon$ denoting the empty word in $F_S$. Whenever a value of
$\Phi_\mm$ is unspecified, we take it to mean $\Phi_\mm(a,s)=(s,a)$.

\begin{itemize}
\item for the states $s_\dagger$ and $t,u$ we put
  \begin{xalignat*}{3}
    \Phi_\mm(0, s_\dagger) &= (\epsilon, \frstate\dagger1); & \Phi_\mm(\frstate\dagger1, x) &= (\epsilon, \frstate\dagger1); & \Phi_\mm(\frstate\dagger1, y) &= (\epsilon, \frstate\dagger1);\\
    \Phi_\mm(\frstate\dagger1, s_\dagger) &= (\epsilon, 0); & \Phi_\mm(\frstate\dagger2, x) &= (\epsilon, \frstate\dagger2); & \Phi_\mm(\frstate\dagger2, y) &= (\epsilon, \frstate\dagger2);\\
    \Phi_\mm(\frstate\dagger1,t) &= (\epsilon, \frstate\dagger2); & \Phi_\mm(0,u)&=(u,1);
    & \Phi_\mm(\frstate\dagger2, s_\dagger) &= (\epsilon, \frstate\dagger2);\\
    \Phi_\mm(\frstate\dagger2,t) &= (\epsilon, \frstate\dagger1) & \Phi_\mm(1,u)&=(u,0).
  \end{xalignat*}
  
\item for all $g\in S\setminus\{u\}$ we put
  $\Phi_\mm(1,g)=(\epsilon,1)$, and for all $a\in A\setminus\{0,1\}$
  we put $\Phi_\mm(a,u)=(\epsilon,a)$.

\item for each instruction $(s_i,m,n)\mapsto(s_j,m+1,n)$ of type I we put
  \begin{xalignat*}{4}
    \Phi_\mm(0, s_i) &= (s_j, \frstate i1); &
    \Phi_\mm(\frstate i1, s_i) &= (\epsilon, 0); &
    \Phi_\mm(\frstate i1, x) &= (x^2, \frstate i1); &
    \Phi_\mm(\frstate i1, y) &= (y, \frstate i1).
  \end{xalignat*}

\item for each instruction $(s_i,m,n)\mapsto(s_j,n,m)$ of type VI,
  $\Phi_\mm(a,s)$ is written at position $(a,s)$ of the following table:
\[\begin{array}{cr|cccccc|}
 & & \multicolumn{6}{c|}{\text{input letter}}\\
 & & 0 & \frstate i1 & \frstate i2 & \frstate i3 & \frstate i4 & \frstate i5\\ \hline
\multirow{5}{*}{\rotatebox{90}{element of $S$}} 
& x & & (x^{b_i x}, \frstate i1) & (\epsilon, \frstate i3) & (x, \frstate i2) & (x^2, \frstate i4) & (y, \frstate i5)\\
& y & & (y^x, \frstate i1) & (y, \frstate i2) & & (y, \frstate i4)  & (x, \frstate i5)\\
& s_i & (a_i b_i x, \frstate i1) & (\epsilon, 0) & & & &\\
& a_i & (a_i, \frstate i2) & & (\epsilon, 0) & & &\\
& b_i & & & (b_i,\frstate i4) & (a_i^{-1}s_j,\frstate i5) & (\epsilon,\frstate i2) & (\epsilon,\frstate i3)\\ \hline
\end{array}
\]

\item for each instruction $(s_i,m,n)\mapsto(m=0\;?\;s_j:s_k,\max(0,m-1),n)$ of
  type IX we put
  \begin{xalignat*}{3}
    \Phi_\mm(0, s_i) &= (s_k, \frstate i1); & \Phi_\mm(\frstate i1,x) &= (s_k^{-1}s_j x,\frstate i2); & \Phi_\mm(\frstate i1,y) &= (y,\frstate i1);\\
    \Phi_\mm(\frstate i1, s_i) &= (\epsilon, 0); & \Phi_\mm(\frstate i2,x) &= (x^{-1} s_j^{-1}s_k x,\frstate i1); &
    \Phi_\mm(\frstate i2,y) &= (y,\frstate i2).
  \end{xalignat*}
\end{itemize}

Theorem~\ref{thm:wp_orbit} follows from the undecidability of the
halting problem for the Minsky machines (Proposition~\ref{prop:minsky})
and the following
\begin{prop}\label{prop:wp_orbit}
  Consider the infinite sequence $W=0^\infty$. Then the Minsky machine
  $\mm$ does not halt if and only if the action of
  $\langle\Phi_\mm\rangle$ satisfies
  \[W^{(s_* x y)t(s_* x y)^{-1}} = W.
  \]
\end{prop}

\begin{proof}
  We encode the states of $\mm$ by elements of $F_S$.  The word
  $(s_i x^{2^m}y^{2^n})\alpha (s_i x^{2^m}y^{2^n})^{-1}$ corresponds to
  the state $(s_i, m, n)$.

  It is convenient to write $\Phi_\mm(a, g) = (g', a')$ in the form
  $a \cdot g = g' \cdot a'$. In this manner, the computation of the
  functionally recursive action is given by a sequence of exchanges of
  letters with words in $F_S$. We check the following equalities:\medskip

  If $(s_i, m, n) \to (s_j, m+1, n)$ is an instruction of type I, then
  \begin{equation}\label{eq:wp_orbit:1}
    0 \cdot (s_i x^{2^m}y^{2^n})t (s_i x^{2^m}y^{2^n})^{-1} = 
    (s_j x^{2^{m+1}}y^{2^n})t (s_j x^{2^{m+1}}y^{2^n})^{-1} \cdot 0.
  \end{equation}
  Indeed
  $0 \cdot s_i x^{2^m}y^{2^n} = s_j \cdot \frstate i1 \cdot x^{2^m}y^{2^n} =
  s_j x^{2^{m+1}}y^{2^n} \cdot \frstate i1$; the claim follows from
  $\frstate i1 \cdot t = t \cdot \frstate i1$ and the reverse
  $ \frstate i1 \cdot (s_i x^{2^m}y^{2^n})^{-1} =
  (s_j x^{2^{m+1}}y^{2^n})^{-1} \cdot 0$.\medskip

  If $(s_i, m, n) \to (s_j, n, m)$ is an instruction of type VI, then
  \begin{equation}\label{eq:wp_orbit:3}
    0^{m+2} \cdot (s_i x^{2^m}y^{2^n})t (s_i x^{2^m}y^{2^n})^{-1} =
    (s_j x^{2^{n}}y^{2^m})t (s_j x^{2^{n}}y^{2^m})^{-1} \cdot 0^{m+2}.
  \end{equation}
  Indeed we first check
  $0 \cdot s_i x^{2^m}y^{2^n} = a_i b_i x(x^{b_i x})^{2^m}(y^x)^{2^n}\cdot\frstate i1=a_i x^{2^m}b_i y^{2^n}x\cdot\frstate i1$.

  We obtained a word with two ``blocks'' of $x$: the blocks $x^{2^m}$
  and $x^{2^0}$. Each time a `$0$' letter is multiplied on the left of
  that word, the size of the first block will halve and the size of
  the second one will double: for $m,n,p\in\N$, we have
  \[0\cdot a_i x^{2m} b_i y^n x^p=a_i x^m b_i y^n x^{2p}\cdot\frstate i4
  \]
  so $0^{m+1}\cdot s_i x^{2^m}y^{2^n} = a_i x b_i
  y^{2^n}x^{2^m}\cdot(\frstate i4)^m \frstate i1$. Then
  \[0\cdot a_i x b_i y^{2^n}x^{2^m}=a_i(a_i^{-1}s_j)x^{2^n}y^{2^m}\cdot\frstate i5,
  \]
  so
  $0^{m+2} \cdot s_i x^{2^m}y^{2^n} = s_j x^{2^n}y^{2^m}\cdot \frstate
  i5 (\frstate i4)^m\frstate i1$.  Recalling that we have
  $a \cdot t = t \cdot a$ for all $a=\frstate i1,\dots,\frstate i5$,
  the claim is proven.\medskip

  If $(s,m,n)\to (m=0\;? \; s_j \; : s_k, \max(m-1,0), n)$ is an instruction of
  type IX, then if $m = 0$ we have
  \begin{equation}\label{eq:wp_orbit:4}
    0\cdot (s_i x^{2^m}y^{2^n})t (s_i x^{2^m}y^{2^n})^{-1} = 
    (s_j x^{2^{m}}y^{2^n})t (s_j x^{2^{m}}y^{2^n})^{-1} \cdot 0
  \end{equation}
  while if $m > 0$ we have
  \begin{equation}\label{eq:wp_orbit:5}
    0\cdot (s_i x^{2^m}y^{2^n})t (s_i x^{2^m}y^{2^n})^{-1} = 
    (s_k x^{2^{m-1}}y^{2^n})t (s_k x^{2^{m-1}}y^{2^n})^{-1} \cdot 0.
  \end{equation}
  Indeed in the first case we have
  \[0\cdot s_i x y^{2^n}=s_k (s_k^{-1}s_j x)y^{2^n}\cdot\frstate i2,\]
  while in the second case we have
  \[0\cdot s_i x^{2^m}y^{2^n}=s_k(s_j^{-1}s_k x\cdot x^{-1}s_k^{-1}s_j
    x)^{2^{m-1}}y^{2^n}\cdot\frstate i1=s_k
    x^{2^{m-1}}y^{2^n}\cdot\frstate i1.\] Recalling that we have
  $a \cdot t = t \cdot a$ for all $a=\frstate i1,\frstate i2$, the
  claim is proven.\medskip

  From~\eqref{eq:wp_orbit:1}--\eqref{eq:wp_orbit:5} it follows
  that if $\mm$ does not halt then $W^{(s_* x y)t(s_* x y)^{-1}} =
  W$. Conversely, if $\mm$ halts then there exist $k,m,n\in\N$ such
  that
\[0^k \cdot (s_* x y)t(s_* x y)^{-1} = 
(s_\dagger x^{2^m}y^{2^n})t (s_\dagger x^{2^m}y^{2^n})^{-1} \cdot 0^k.
\]
Then
\[
  0 \cdot s_\dagger x^{2^m}y^{2^n}t (s_\dagger x^{2^m}y^{2^n})^{-1}=
  \frstate\dagger1 \cdot t (s_\dagger x^{2^m}y^{2^n})^{-1} =
  \frstate\dagger2 \cdot (s_\dagger x^{2^m}y^{2^n})^{-1} =
  \frstate\dagger2.
\]
In that case, we have
$W^{(s_* x y)t(s_* x y)^{-1}} = 0^{k}\frstate\dagger2 0^{\infty}\neq W$.
\end{proof}

The computations are best carried on $\Phi_\mm$'s \emph{dual Moore
  diagram} $\Delta$, see Figure~\ref{fig:wpmoore} : this is the
directed labeled graph with vertex set $A$ and with for all
$a\in A,s\in S$ an edge from $a$ to $b$ labeled $(s,t)$ whenever
$\Phi_\mm(a,s)=(t,b)$. One checks an equality `$\Phi_\mm(a,s)=(t,b)$'
by finding in $\Delta$ a path starting at $a$ with input label $s$;
the endpoint of the path is $b$, and the output label is $t$.

\begin{figure}[h]
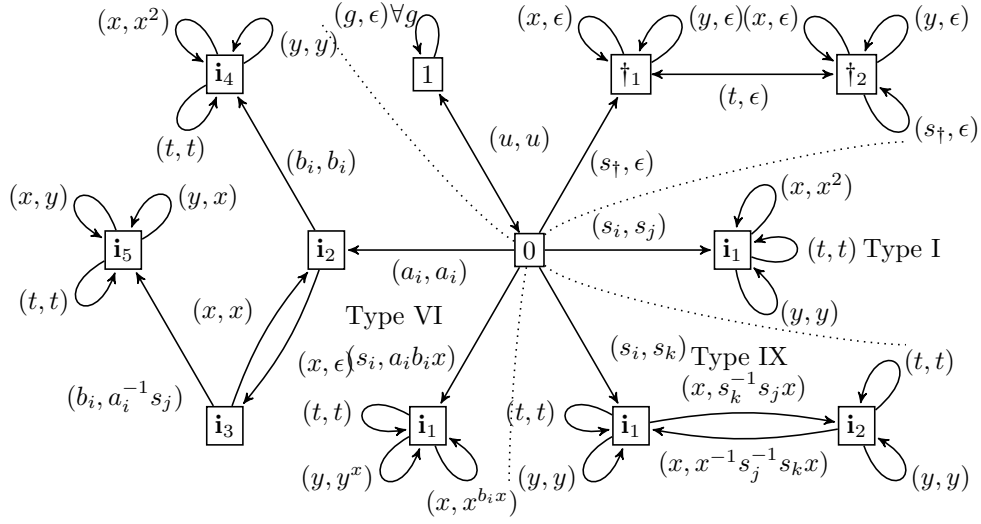

  \[\begin{dualmoore}[scale=0.9,node distance=3cm]
    \node[state] (0) at (0,0) {$0$};
    \node[state] (1) at (120:3) {$1$};
    \node[state] (d1) at (60:3) {$\frstate\dagger1$};    
    \node[state] (d2) [right of=d1] {$\frstate\dagger2$};    
    \node[state] (I1) at (0:3) {$\frstate i1$};
    \node[state] (IX1) at (-60:3) {$\frstate i1$};
    \node[state] (IX2) [right of=IX1] {$\frstate i2$};
    \node[state] (VI1) at (-120:3) {$\frstate i1$};
    \node[state] (VI2) at (180:3) {$\frstate i2$};
    \node[state] (VI3) at ($(VI2) +(240:3)$) {$\frstate i3$};
    \node[state] (VI4) at ($(VI2) +(120:3)$) {$\frstate i4$};
    \node[state] (VI5) at ($(VI3) +(120:3)$) {$\frstate i5$};

    \node at (5.5,0) {Type I};
    \node at (3.1,-1.6) {Type IX};
    \node at (-2,-1) {Type VI};

    \draw[-,dotted] (0) .. controls +(45:1) and +(180:1) .. (5.6,1.6);
    \draw[-,dotted] (0) .. controls +(150:1) and +(-60:1) .. (-3.1,3.4);
    \draw[-,dotted] (0) .. controls +(-45:1) and +(180:1) .. (5.6,-1.4);
    \draw[-,dotted] (0) .. controls +(-100:1) and +(90:1) .. (-0.3,-3.5);

    \path (0) edge[<->] node[above right] {$(u,u)$} (1)
              edge node[right] {$(s_\dagger,\epsilon)$} (d1)
              edge node {$(s_i,s_j)$} (I1)
              edge node [near end] {$(s_i,s_k)$} (IX1)
              edge node [pos=0.66,left] {$(s_i,a_i b_i x)$} (VI1)
              edge node {$(a_i,a_i)$} (VI2)
          (1) edge [in=110,out=70,loop] node[left] {$(g,\epsilon)\forall g$} ()
         (d1) edge [<->] node[below] {$(t,\epsilon)$} (d2)
              edge [in=150,out=110,loop] node[left] {$(x,\epsilon)$} ()
              edge [in=70,out=30,loop] node[right] {$(y,\epsilon)$} ()
         (d2) edge [in=150,out=110,loop] node[left] {$(x,\epsilon)$} ()
              edge [in=70,out=30,loop] node[right] {$(y,\epsilon)$} ()
              edge [in=-30,out=-70,loop] node[right] {$(s_\dagger,\epsilon)$} ()
         (I1) edge [in=80,out=40,loop] node[right] {$(x,x^2)$} ()
              edge [in=-40,out=-80,loop] node[right] {$(y,y)$} ()
              edge [in=20,out=-20,loop] node[right] {$(t,t)$} ()
        (IX1) edge [in=190,out=150,loop] node[left] {$(t,t)$} ()
              edge [in=250,out=210,loop] node[left] {$(y,y)$} ()
              edge [bend left=10] node {$(x,s_k^{-1}s_j x)$} (IX2)
        (IX2) edge [in=80,out=40,loop] node[right] {$(t,t)$} ()
              edge [bend left=10] node {$(x,x^{-1}s_j^{-1}s_k x)$} (IX1)
              edge [in=-30,out=-70,loop] node[right] {$(y,y)$} ()
        (VI1) edge [in=190,out=150,loop] node[left] {$(t,t)$} ()
              edge [in=-30,out=-70,loop] node[below] {$(x,x^{b_i x})$} ()
              edge [in=250,out=210,loop] node[left] {$(y,y^x)$} ()
        (VI2) edge [bend left=10] node {$(x,\epsilon)$} (VI3)
              edge node [right] {$(b_i,b_i)$} (VI4)
        (VI3) edge [bend left=10] node {$(x,x)$} (VI2)
              edge node [near start] {$(b_i,a_i^{-1}s_j)$} (VI5)
        (VI4) edge [in=150,out=110,loop] node[left] {$(x,x^2)$} ()
              edge [in=70,out=30,loop] node[below right] {$(y,y)$} ()
              edge [in=250,out=210,loop] node[below] {$(t,t)$} ()
        (VI5) edge [in=150,out=110,loop] node[left] {$(x,y)$} ()
              edge [in=70,out=30,loop] node[right] {$(y,x)$} ()
              edge [in=250,out=210,loop] node[left] {$(t,t)$} ();
    \end{dualmoore} 
  \]
  \caption{The dual Moore diagram of $\Phi_\mm$, used in the proof of
    Theorem~\ref{thm:wp}}
  \label{fig:wpmoore}
\end{figure}

\subsection{Proof of Theorem~\ref{thm:wp}}
We have not yet used the letter $1$ and the state $u$ of
$\Phi_\mm$. Theorem~\ref{thm:wp} follows now from the following
\begin{prop}\label{prop:wp}
  The Minsky machine $\mm$ halts if and only if
  \[[(s_* x y)t(s_* x y)^{-1},u]\neq 1 \text{ in } \langle\Phi_\mm\rangle.
  \]
\end{prop}
\begin{proof}
  The element $u$ acts on $A^\omega$ as follows: it scans
  $X\in A^\omega$ for its longest prefix in $\{0,1\}^*$, and exchanges
  all $0$ and $1$ in that prefix. Write $g=(s_* x y)t(s_* x y)^{-1}$;
  from Proposition~\ref{prop:wp_orbit} we know that $g$ fixes
  $0^\infty$ if and only if $\mm$ does not halt.

  Assume first that $\mm$ does not halt; then $g$ in fact also fixes
  $\{0,1\}^\infty$, so the supports of $g$ and $u$ are disjoint and
  $[g,u]=1$ in $\langle\Phi_\mm\rangle$.

  Assume next that $\mm$ does halt; without loss of generality, we may
  assume $\mm$ does not stop immediately, so there is $k\ge1$ such
  that $(0^{k+1})^g=0^k\frstate\dagger2$. Since $(0^{k+1})^u=1^{k+1}$
  and $(0^k\frstate\dagger2)^u=1^k\frstate\dagger2$ and
  $(1^{k+1})^g=1^{k+1}$ and
  $(1^k\frstate\dagger2)^g=1^k\frstate\dagger2$, the commutator
  $[g,u]$ acts as a $2$-$2$-cycle
  $(0^{k+1},0^k\frstate\dagger2)(1^{k+1},1^k\frstate\dagger2)$ and in
  particular $[g,u]\neq1$ in $\langle\Phi_\mm\rangle$:
  \[\begin{dualmoore}[node distance=2cm,baseline=-3mm]
      \node[state] (11) at (0,0) {$1^{k+1}$};
      \node[state] (00) [right of=11] {$0^{k+1}$};
      \node[state] (0d) [right of=00] {$0^k\frstate\dagger2$};
      \node[state] (1d) [right of=0d] {$1^k\frstate\dagger2$};
      \path (11) edge [in=195,out=165,loop] node [left] {$g$} ()
                 edge[<->] node {$u$} (00)
            (00) edge[<->] node {$g$} (0d)
            (1d) edge[<->] node [above] {$u$} (0d)
                 edge [in=15,out=-15,loop] node [right] {$g$} ();
  \end{dualmoore}\qedhere\]
\end{proof}

\subsection{Proof of Theorem~\ref{thm:wp_uniform}}
Consider a universal Minsky machine $\mm_u$, namely one that emulates
arbitrary Turing machines encoded in an integer $n$, when started in
state $(s_*,0,n)$. The set of $2^n$ such that $\mm_u$ halts when
started in state $(s_*, 0, n)$ is not recursive, see
Proposition~\ref{prop:minsky}. Therefore, Theorem~\ref{thm:wp_uniform}
follows by considering in the group $\langle\Phi_{\mm_u} \rangle$ the
elements $[(s_* x y^{2^n})t(s_* x y^{2^n})^{-1},u]$; this set of words
is recursive, but the subset of those that equal $1$ in
$\langle\Phi_{\mm_u} \rangle$ is not recursive.

\section{Automata groups and Minsky machines}

\subsection{Proof of Theorem~\ref{thm:torsion}}

\def\state#1#2{\text{\uppercase\expandafter{\romannumeral #1}}_{#2}}

Let $\mm$ by a Minsky machine with stateset $S_0$. Without loss of
generality, we assume that all instructions of $\mm$ are of type
III,IV,V,VII,VIII, as defined in the beginning of Section
\ref{sec:wp},
\[S_0 = S_{\state3{}} \sqcup  S_{\state4{}} \sqcup  S_{\state5{}} \sqcup  S_{\state7{}} 
\sqcup  S_{\state8{}} \sqcup \{s_\dagger\}
\]

We consider the transducer with stateset
$S\coloneqq S_0^{\pm1}\sqcup\{\epsilon,x,x^{-1},y,y^{-1}\}$ and
alphabet
\[A=\{\state3i,\state4i,\state5i,\state7j,\state8j\mid i=1,2,\overline1,\overline2;\;j=1,\dots,4,\overline1,\dots,\overline4\}.
\]
The structure of the transducer is given by its map
$\Phi_\mm\colon A\times S\to S\times A$, first described as a table, with
$\Phi_\mm(a,s)$ at position $(a,s)$. The state $\epsilon$ is the identity,
and $\Phi_\mm(a,\epsilon)=(\epsilon,a)$ for all $a\in A$.

\noindent For all instructions $(s,m,n)\mapsto(s',m+1,n+1)$ of type III and for all $t\in S_0\setminus S_{\state3{}}$ we have 
\[\begin{array}{cr|cccc|}
 & & \multicolumn{4}{c|}{\text{input letter}}\\
 & & \state31 & \state32 & \state3{\overline1} & \state3{\overline2}\\ \hline
\multirow{4}{*}{\rotatebox{90}{in state}} & x & (x,\state31) & (x,\state32) & (x^{-1},\state3{\overline1}) & (x^{-1},\state3{\overline2})\\
 & y & (y,\state31) & (y,\state32) & (y^{-1},\state3{\overline1}) & (y^{-1},\state3{\overline1})\\
 & s & (s',\state32) & (\epsilon,\state31) & (\epsilon,\state3{\overline2}) & ((s')^{-1},\state3{\overline1})\\
 & t & (\epsilon,\state3{\overline1}) & (\epsilon,\state3{\overline2}) & (\epsilon,\state31) & (\epsilon,\state32)\\ \hline
\end{array}
\]

\noindent For all instructions $(s,m,n)\mapsto(s',m-1,n)$ of type IV
and for all $t\in S_0\setminus S_{\state4{}}$, we have
\[\begin{array}{cr|cccc|}
 & & \multicolumn{4}{c|}{\text{input letter}}\\
 & & \state41 & \state42 & \state4{\overline1} & \state4{\overline2}\\ \hline
\multirow{4}{*}{\rotatebox{90}{in state}} & x & (x,\state42) & (\epsilon,\state41) & (\epsilon,\state4{\overline2}) & (x^{-1},\state4{\overline1})\\
 & y & (y,\state41) & (y,\state42) & (y^{-1},\state4{\overline1}) & (y^{-1},\state4{\overline2})\\
 & s & (s',\state41) & (s',\state42) & ((s')^{-1},\state4{\overline1}) & ((s')^{-1},\state4{\overline2})\\
 & t & (\epsilon,\state4{\overline1}) & (\epsilon,\state4{\overline2}) & (\epsilon,\state41) & (\epsilon,\state42)\\ \hline
\end{array}
\]
The same applies for an instruction of type V, with the roles of $x,y$
switched.
  
For an instruction $(s,m,n)\mapsto(m=0\;?\;s':s'',m,n)$ of type
VII and for all $t\in S_0\setminus S_{\state7{}}$, we have 
\[\kern-20mm\begin{array}{cr|cccccccc|}
 & & \multicolumn{8}{c|}{\text{input letter}}\\
 & & \state71 & \state72 & \state73 & \state74 & \state7{\overline1} & \state7{\overline2} & \state7{\overline3} & \state7{\overline4}\\ \hline
\multirow{4}{*}{\rotatebox{90}{in state}} & x & (x,\state74) & (\epsilon,\state73) & (\epsilon,\state72) & (x,\state71) & (x^{-1},\state7{\overline4}) & (\epsilon, \state7{\overline3}) & (\epsilon,\state7{\overline2}) & (x^{-1},\state7{\overline1})\\
 & y & (y,\state71) & (\epsilon,\state72) & (\epsilon,\state73) & (y,\state74) & (y^{-1},\state7{\overline1}) & (\epsilon,\state7{\overline2}) & (\epsilon,\state7{\overline3}) & (y^{-1},\state7{\overline4})\\
 & s & (\epsilon,\state72) & (s'',\state71) & (s',\state74) & (\epsilon,\state7{\overline4}) & ((s'')^{-1},\state7{\overline2}) & (\epsilon,\state7{\overline1}) & (\epsilon,\state7{3}) & ((s')^{-1},\state7{\overline3})\\
 & t & (\epsilon,\state7{\overline 1}) & (\epsilon,\state7{\overline2}) & (\epsilon,\state7{\overline3}) & (\epsilon,\state7{\overline 4}) & (\epsilon,\state7{1}) & (\epsilon,\state7{2}) & (\epsilon,\state7{3}) & (\epsilon,\state7{4})\\ \hline
\end{array}
\]
The same applies for an instruction of type VIII, with the roles of
$x,y$ switched.

\noindent Note that $s_\dagger$ is treated as a state $t$ in all
tables above.

Theorem~\ref{thm:torsion} follows from the undecidability of the
halting problem for Minsky machines, and the following
\begin{prop}
  The Minsky machine $\mm$ constructed above halts if and only
  if the element $s_* x y$ has finite order in
  $\langle\Phi_\mm\rangle$.
\end{prop}
\begin{proof}
  Set $G=\langle\Phi_\mm\rangle$. For $g\in G$, denote by $C(g)$ its
  symmetric conjugacy class:
  \[C(g)\coloneqq\{g^{\pm x}\mid x\in G\}.
  \]

  Given a symmetric conjugacy class $C$, choose a representative $g$
  in it, let $A=A_1\sqcup\dots\sqcup A_\ell$ be the decomposition of
  $A$ into cycles for the action of $g$, and choose representatives
  $a_i\in A_i$. We have
  $\overline\Phi_\mm(a_i,g^{\#A_i})=(h_i,a_i)$ for some
  $h_i\in G$, and it is easy to see that the collection of symmetric
  conjugacy class $\{C(h_i)\mid i=1,\dots,\ell\}$ is independent of
  the choice of $g$ and the $a_i$.

  We construct an integer-labeled, directed graph\footnote{This graph
    essentially appears in the solution
    of~\cite{bondarenko-b-s-z:conjugacy} to the order problem in
    bounded automata.} whose vertices are symmetric conjugacy classes
  in $G$; for a conjugacy class $C$ as above, there are $\ell$ edges
  starting at $C$, ending respectively at $C(h_1),\dots,C(h_\ell)$
  with labels $\#A_1,\dots,\#A_\ell$.

  \begin{lem}
    For $g\in G$, its order (in $\N\cup\{\infty\}$) is the least
    common multiple, along all paths starting at $C(g)$, of the
    product of the labels along the path.
  \end{lem}
  \begin{proof}
    Consider a path starting at $C(g)$, with labels $n_1,\dots,n_s$,
    and going through vertices $C(g_1),\dots,C(g_s)$. Then $g$ has an
    orbit of length $n_1$ on $A$, so the order of $g$ is a multiple of
    $n_1$. Furthermore, $g^{n_1}$ fixes that orbit, and acts as $g_1$
    on sequences that start by that orbit. Recursively, the order of
    $g_1$ is a multiple of $n_2\cdots n_s$, so the order of $g$ is a
    multiple of $n_1\cdots n_s$. In particular, if there are paths
    with arbitrarily large product of labels then $g$ has infinite
    order.

    Conversely, if $g$ has infinite order then there are arbitrarily
    long orbits of $g$ on $A^*$, so there are paths with arbitrarily
    large product of labels; and if $m$ be the least common multiple
    of all path labels then all edges on paths starting at $C(g^m)$
    are labeled $1$ so $g^m$ fixes every sequence and therefore
    $g^m=1$.
  \end{proof}
  
  Let us compute the subgraph spanned by $C(s_* x y)$. For the
  computations, it is helpful to picture the operation of the
  transducer $\Phi_\mm$ by means of its dual Moore diagram $\Delta$,
  see Figure~\ref{fig:torsionmoore}. Given $g\in G$, we compute all
  primitive cycles in $\Delta$ whose input label is a power of $g$,
  and read the corresponding output label; these are the $h_i$ in the
  map on symmetric conjugacy classes $C(g)\rightsquigarrow\{C(h_i)\}$.

  \begin{figure}[h]
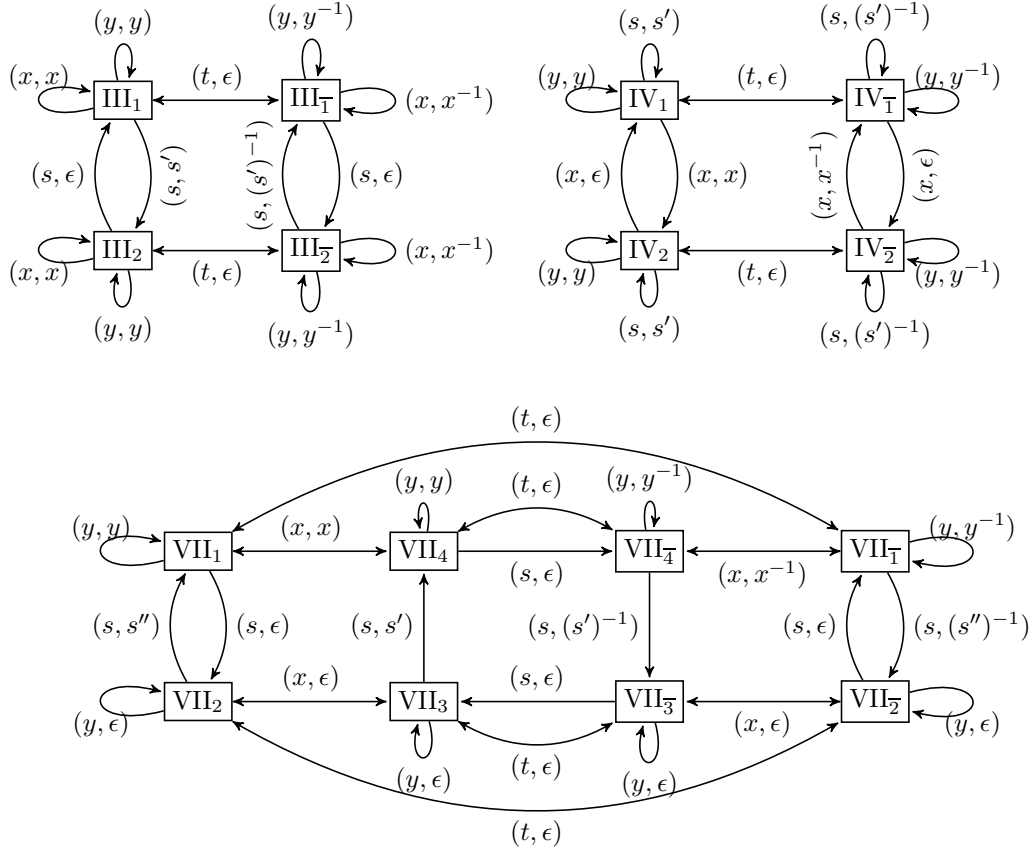

    \[\begin{dualmoore}
      \node[state] (3_1) at (-2,0) {$\state31$};
      \node[state] (3_2) [below of=3_1] {$\state32$};
      \node[state] (3_1x) [node distance=2.5cm, right of=3_1] {$\state3{\overline1}$};
      \node[state] (3_2x) [node distance=2.5cm, right of=3_2] {$\state3{\overline2}$};

      \node[state] (4_1) at (5,0) {$\state41$};
      \node[state] (4_2) [below of=4_1] {$\state42$};
      \node[state] (4_1x) [node distance=3cm,right of=4_1] {$\state4{\overline1}$};
      \node[state] (4_2x) [below of=4_1x] {$\state4{\overline2}$};

      \node[state] (7_1) at (-1,-6) {$\state71$};
      \node[state] (7_2) [below of=7_1] {$\state72$};
      \node[state] (7_3) [node distance=3cm,right of=7_2] {$\state73$};
      \node[state] (7_4) [above of=7_3] {$\state74$};
      \node[state] (7_3x) [node distance=3cm,right of=7_3] {$\state7{\overline3}$};
      \node[state] (7_4x) [above of=7_3x] {$\state7{\overline4}$};
      \node[state] (7_1x) [node distance=3cm,right of=7_4x] {$\state7{\overline1}$};
      \node[state] (7_2x) [below of=7_1x] {$\state7{\overline2}$};

      \path (3_1) edge [bend left] node {\rotatebox{90}{$(s,s')$}} (3_2)
                  edge [loop left] node [above] {$(x,x)$} (3_1)
                  edge [loop above] node [above] {$(y,y)$} (3_1)
                  edge [<->] node {$(t,\epsilon)$} (3_1x)
            (3_2) edge [bend left] node {$(s,\epsilon)$} (3_1)
                  edge [loop left] node [below] {$(x,x)$} (3_2)
                  edge [loop below] node [below] {$(y,y)$} (3_2)
                  edge [<->] node [below] {$(t,\epsilon)$} (3_2x)
		(3_1x) edge [bend left] node {$(s,\epsilon)$} (3_2x)
			edge [loop above] node {$(y, y^{-1})$} ()
			edge [loop right] node {$(x, x^{-1})$} ()
		(3_2x) edge [bend left] node {\rotatebox{90}{$(s, (s')^{-1})$}} (3_1x)
			edge [loop below]  node {$(y, y^{-1})$} () 
			edge [loop right]  node {$(x, x^{-1})$} ();

      \path (4_1) edge [loop above] node {$(s,s')$} ()
                  edge [bend left] node {$(x,x)$} (4_2)
                  edge [loop left] node [above] {$(y,y)$} ()
                  edge [<->] node {$(t,\epsilon)$} (4_1x)
            (4_2) edge [loop below] node {$(s,s')$} ()
                  edge [bend left] node {$(x,\epsilon)$} (4_1)
                  edge [loop left] node [below] {$(y,y)$} ()
                  edge [<->] node [below] {$(t,\epsilon)$} (4_2x)
            (4_1x) edge [loop above] node {$(s,(s')^{-1})$} ()
                  edge [bend left] node {\rotatebox{90}{$(x,\epsilon)$}} (4_2x)
                  edge [loop right] node [above] {$(y,y^{-1})$} ()
            (4_2x) edge [loop below] node {$(s,(s')^{-1})$} ()
                  edge [bend left] node {\rotatebox{90}{$(x,x^{-1})$}} (4_1x)
                  edge [loop right] node [below] {$(y,y^{-1})$} ();

      \path (7_1) edge [bend left] node {$(s,\epsilon)$} (7_2)
                  edge [<->] node {$(x,x)$} (7_4)
                  edge [loop left] node [above] {$(y,y)$} ()
                  edge [<->, bend left] node {$(t,\epsilon)$} (7_1x)
            (7_2) edge [bend left] node {$(s,s'')$} (7_1)
                  edge [<->] node {$(x,\epsilon)$} (7_3)
                  edge [loop left] node [below] {$(y,\epsilon)$} ()
                  edge [<->, bend right] node [below] {$(t,\epsilon)$} (7_2x)
            (7_3) edge [] node {$(s,s')$} (7_4)
                  edge [<->, bend right] node [below] {$(t,\epsilon)$} (7_3x)
                  edge [loop below] node [below] {$(y,\epsilon)$} ()
            (7_4) edge [] node [below] {$(s,\epsilon)$} (7_4x)
                  edge [<->, bend left] node {$(t,\epsilon)$} (7_4x)
                  edge [in=100,out=80,loop] node [above] {$(y,y)$} ()
            (7_1x) edge [bend left] node {$(s,(s'')^{-1})$} (7_2x)
                  edge [<->] node {$(x,x^{-1})$} (7_4x)
                  edge [loop right] node [above] {$(y,y^{-1})$} ()
            (7_2x) edge [bend left] node {$(s,\epsilon)$} (7_1x)
                  edge [<->] node {$(x,\epsilon)$} (7_3x)
                  edge [loop right] node [below] {$(y,\epsilon)$} ()
            (7_3x) edge [] node [above] {$(s,\epsilon)$} (7_3)
                  edge [loop below] node [below] {$(y,\epsilon)$} ()
            (7_4x) edge [] node [left] {$(s,(s')^{-1})$} (7_3x)
                  edge [in=80,out=100,loop] node [above] {$(y,y^{-1})$} ();
    \end{dualmoore}
    \]
    \caption{The dual Moore diagram of $\Phi_\mm$, used in the proof of
      Theorem~\ref{thm:torsion}}
    \label{fig:torsionmoore}
  \end{figure}

  We first note, by direct inspection, that $x$ and $y$ commute in
  $G$. This follows by tracing the path $x^{-1}y^{-1}x y$ in the
  graphs above, and noting that they always induce the trivial
  permutation of $A$ with output either trivial or conjugate to
  $(x^{-1}y^{-1}x y)^{\pm1}$.

  We now claim that, if $(s,m,n)\to(s',m',n')$ is a transition of the
  machine $\mm$, then the conjugacy class $C(s x^{2^m} y^{2^n})$ has
  at least one arrow to $C(s' x^{2^{m'}} y^{2^{n'}})$, and possibly
  other arrows, all of them to $C(1)$.  We also claim that if $s$ is
  not of type IV or V, then arrows to $C(s' x^{2^{m'}} y^{2^{n'}})$
  are with labels $>1$; and all arrows from
  $C(s_\dagger x^{2^m}y^{2^n})$ are arrows to $C(1)$.  We see that if
  the machine halts then every path starting at $C(s_* x y)$ has only a
  finite number of labels $>1$, and this shows that the order of
  $s_* x y$ is finite.

  On the other hand, if the machine does not halt then there is an
  path with infinitely many labels $>1$ (because no Minsky machine can
  decrease its counters infinitely many times in a row) so $s_* x y$
  has infinite order.
 
  Note that our transducer has the property
  $\Phi_{\mm}(L_{\overline i}, g) = \Phi_{\mm}(L_i, g^{-1})$, for all
  $g \in S$ and all
  $L\in \{\state3{}, \state4{}, \state5{}, \state7{}, \state8{} \}$.
  Also note that
  $\Phi_{\mm}(L_i, t) = \Phi_{\mm}(L_i, t^{-1}) = (\epsilon,
  L_{\overline i})$ whenever $t$ is any instruction not of type $L$.

  Using this, we can prove that if $t$ is not of type $L$, then
  $t g_n g_{n-1}\cdots g_1 t g_1 g_2 \cdots g_n$ fixes the orbit
  $\{L_i\}$ with output $\epsilon$.

  Indeed,
  $t g_n g_{n-1}\cdots g_1 t g_1 g_2 \cdots g_n = (g_n t g_n t^{-1} \cdot
  t g_{n-1}g_{n-2}\cdots g_1 t g_1 g_2 \cdots g_{n-1})^{g_n}$, and we
  use induction on $n$.  It follows that $(t x^my^n)^2$ fixes ${L_i}$
  with outputs $\epsilon$, i.e., there is an arrow from $C(t x^my^n)$ to
  $C(1)$ with label 2.
 
  Let us first restrict to the orbit $\{\state3i\}$ of $G$ on $A$.  We
  consider $g\coloneqq s x^m y^n$ with $s$ an instruction of type
  III. It acts as a product of two cycles
  $(\state31,\state32)(\state3{\overline 1}, \state3{\overline2})$;
  the output label of $g^2$ on the first cycle, starting at
  $\state31$, is $s' x^m y^n\epsilon x^m y^n=s' x^{2m} y^{2n}$, and
  the output of $g^2$ starting on the second cycle at
  $\state3{\overline 1}$ is
  $\epsilon x^{-m} y^{-n}(s')^{-1}x^{-m}y^{-n}\in C(s' x^{2m}
  y^{2n})$.  There are therefore two arrows from
  $C(s x^{2^{m}} y^{2^n})$ to $C(s' x^{2^{m+1}} y^{2^{n+1}})$, as
  required.  We do not consider $g\coloneqq t x^m y^n$ with $t$ an
  instruction of different type or $s_\dagger$, because it was
  considered above (there are some arrows to $C(1)$ with labels $2$).

  We restrict next to the orbit $\{\state4i\}$ of $G$ and consider the
  case $g=s x^{2m} y^n$. (We do not need to consider cases
  $g = s x^{2m + 1}y^n$ or $g=t x^{m} y^n$, for the first because we
  suppose that if $m = 0$ then $\mm$ does not perform an instruction
  of type IV, and for the second because it was already considered
  above.)
 
  An element $g = s x^{2m} y^n$ fixes $\state4{1}$, $\state4{2}$,
  $\state4{\overline{1}}$ and $\state4{\overline2}$, its outputs are
  respectively $s'(x\epsilon)^m y^n = s' x^m y^n$, $s' x^my^n$,
  $(s')^{-1}x^{-m}y^{-n}$ and $(s')^{-1}x^{-m}y^{-n}$.  Hence there
  are $4$ arrows from $C(s x^{2^m}y^{2^n})$ to $C(s' x^{2^{m-1}}y^{2^n})$,
  all with labels $1$.
 
  We restrict next to the orbit $\{\state7i\}$ of $G$ and perform the
  same computations, the result is in the following table:
 \[\begin{array}{c|c|l|c|}
     g\in G & \text{cycles of } g &\text{output, starting at first element of the cycle} \\
     \hline
     \multirow{3}{*}{$s x^{2m}y^n$} & (\state7{1},\state7{2})& \epsilon s''x^{2m}y^n \\
            & (\state7{3},\state7{4},\state7{\overline4},\state7{\overline3})& s' x^{2m}y^n\epsilon x^{-2m}y^{-n}(s')^{-1}\epsilon = 1\\
            & (\state7{\overline1},\state7{\overline2}) & (s'')^{-1}\epsilon x^{-2m}y^{-n}\\
     \hline
\multirow{3}{*}{$s x^{2m+1}y^n$} & (\state7{1},\state7{3})& \epsilon s' x^{2m+1}y^n  \\
& (\state7{2},\state7{4},\state7{\overline1},\state7{\overline3}) & s'' x^{2m + 1} y^{n} \epsilon x^{-2m-1} y^{-n}(s'')^{-1} \epsilon = 1\\
&   (\state7{\overline2},\state7{\overline4}) & \epsilon x^{-2m-1}y^{-n}(s')^{-1}\epsilon\\

\hline
\end{array}
\]

If $m>0$ then there are therefore two arrows from $C(s x^{2^m}y^{2^n})$
to $C(s''x^{2^m}y^{2^n})$ with labels 2 and an arrow to $C(1)$ with
label $4$; if $m=0$ then there are two arrows from
$C(s x^{2^m}y^{2^n})$ to $C(s''x^{2^m}y^{2^n})$ with label 4 and an
arrow to $C(1)$ with label $4$.
 
The orbits $\{\state5i\}$ and $\{\state8i\}$ are investigated in the
same way as $\{\state4i\}$ and $\{\state7i\}$ respectively.
\end{proof}

\subsection{Proof of Theorem~\ref{thm:torsion_uniform}}
Consider a universal Minsky machine $\mm_u$, namely one that emulates
arbitrary Turing machines encoded in an integer $n$, when started in
state $(s_*,0,n)$. The set of $2^n$ such that $\mm_u$ halts when
started in state $(s_*, 0, n)$ is not recursive, see
Proposition~\ref{prop:minsky}. Therefore,
Theorem~\ref{thm:torsion_uniform} follows by considering in the group
$\langle\Phi_{\mm_u} \rangle$ the elements $s = s_* x$ and $t = y$.

\subsection{Proof of Theorem~\ref{thm:torsion_orbit}}
Let $\mm$ be a Minsky machine with stateset $S_0$. Without loss of
generality, we assume that all instructions are of type III,IX,X.

We associate to it the transducer with stateset
$Q:=S_0\sqcup\{\epsilon, x,y\}$ and alphabet
\[A =\{0, \state3i, \state9j, \state{10}j \mid i = 1, 2; \; j = 1,
\dots, 4\}.\]

The structure of the transducer is given by its map
$\Phi_\mm\colon A\times Q \to Q\times A$.

The state $\epsilon$ is the identity, and
$\Phi_\mm (a, \epsilon) = (\epsilon, a)$ for all $a\in A$.

\begin{itemize}
\item for all instructions $(s_i,m,n)\mapsto(s'_i,m+1,n+1)$ of type III we have
  \[\begin{array}{cr|ccc|}
    & & \multicolumn{3}{c|}{\text{input letter}}\\
      & & 0 & \state31 & \state32 \\ \hline
      \multirow{3}{*}{\rotatebox{90}{in state}} 
      & x & (\epsilon, 0) & (x, \state31) & (x, \state32)\\
      & y & (\epsilon, 0) & (y, \state31) & (y, \state32) \\
      & s_i & (s_i', \state31) & (\epsilon, \state32) & (\epsilon, 0)\\
      \hline
    \end{array}\]
  and every instruction $t$ of another type acts as $\Phi_\mm(t, \state3\ell) = (\state3\ell, t)$.

\item for all instructions
  $(s_j,m,n)\mapsto m=0\;?\;(s'_j,m,n):(s''_j,m-1,n)$ of type IX we have
\[\begin{array}{cr|ccccc|}
 & & \multicolumn{5}{c|}{\text{input letter}}\\
 & & 0 & \state91 & \state92 & \state93 & \state94 \\ \hline
\multirow{3}{*}{\rotatebox{90}{in state}} 
& x & (\epsilon, 0)& (\epsilon, \state92)& (\epsilon, \state91)& (x, \state94)& (\epsilon, \state93)\\
& y & (\epsilon, 0)& (\epsilon, \state91)& (\epsilon, \state92)& (y, \state93)& (y, \state94)\\
& s_j & (\epsilon, \state91)& (s''_j, \state94)& (s'_j, \state93)& (s'_j, \state92)& (\epsilon, 0)\\
\hline
  \end{array}\]
and every instruction $t$ of another type acts as $\Phi_\mm(\state9\ell, t) = (t, \state9\ell)$.

\item The same applies for every instruction
  $(s_k,m,n)\mapsto n=0\;?\;(s_k',m,n):(s_k'',m,n-1)$ of type X, with
  the roles of $x$ and $y$ switched.

\item for all $a \in A$ we have $\Phi_\mm(a, s_\dagger) = (\epsilon, a)$.
\end{itemize}

We claim that the orbit of $0^\infty$ under $s_* x y$ is finite if and
only if the machine $\mm$ stops.

Set $G = \langle\Phi_\mm\rangle$.  We construct an integer-labeled,
directed graph whose vertices are elements of $G$.  For $g\in G$
consider its action on $A$ and the minimal $p_g$ such that $g^{p_g}$
fixes $0$, i.e., $0\cdot g^{p_g} = g' \cdot 0$. In our graph we put an
edge $g\to g'$ with label $p_g$ on it.

The size of the the orbit of $0^\infty$ under $s_* x y$ is a finite
number or $\infty$ and it is equal to product of the labels along the
path starting at $s_* x y$.

We claim that for any instruction $(s, m, n) \to (s', m', n')$ there
is an edge from $s x^{2^m}y^{2^n}$ to $s' x^{2^{m'}}y^{2^{n'}}$ with
label 3, and an edge from $s_\dagger x^{2^m}y^{2^n}$ to $1$.  This is
checked on the dual Moore diagram of $\Phi_\mm$, see
Figure~\ref{fig:torsion_orbitmoore}:

\begin{figure}[h]
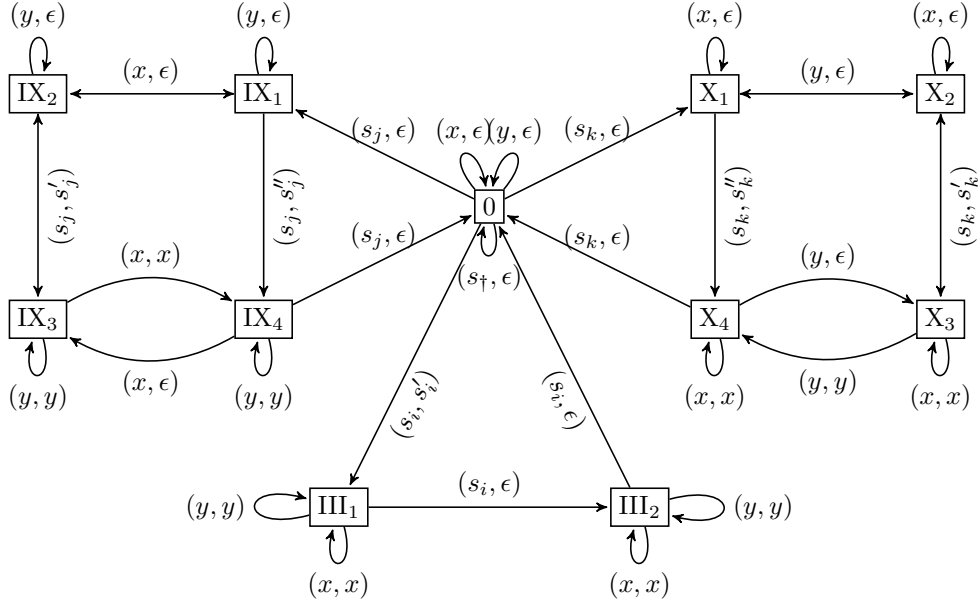

  \[\begin{dualmoore}
\node[state] (0) at (0,0) {0};

\node[state] (9_1) at (-3,1.5) {$\state91$};
\node[state] (9_2) at (-6,1.5) {$\state92$};
\node[state] (9_3) at (-6,-1.5) {$\state93$};
\node[state] (9_4) at (-3,-1.5) {$\state94$};

\node[state] (10_1) at (3,1.5) {$\state{10}1$};
\node[state] (10_2) at (6,1.5) {$\state{10}2$};
\node[state] (10_3) at (6,-1.5) {$\state{10}3$};
\node[state] (10_4) at (3,-1.5) {$\state{10}4$};

\node[state] (3_1) at (-2,-4) {$\state{3}1$};
\node[state] (3_2) at (2,-4) {$\state{3}2$};

\path (0) edge [in=80,out=50,loop] node [above] {$(y,\epsilon)$} ()
 edge [in=100,out=130,loop] node [above] {$(x,\epsilon)$} ()
 edge node [below] [below] {\rotatebox{65}{$(s_i, s_i')$}}  (3_1)
 edge [loop below] node {$(s_\dagger, \epsilon)$} ()
 (3_1) edge node [above] {$(s_i, \epsilon)$} (3_2)
   edge [loop left] node [] {$(y,y)$} () 
   edge [loop below] node [] {$(x,x)$} ()
 (3_2) edge node [below] {\rotatebox{-65}{$(s_i, \epsilon)$}} (0)
  edge [loop right] node [] {$(y,y)$} () 
   edge [loop below] node [] {$(x,x)$} () ;
 
 \path (0) edge  node [above] {$(s_k, \epsilon)$} (10_1)
 (10_1) edge [<->] node [above] {$(y, \epsilon)$} (10_2)
 edge node [] {\rotatebox{90}{$(s_k, s_k'')$}} (10_4)
  edge [loop above] node [above] {$(x,\epsilon)$} ()
 (10_2) edge [<->] node [] {\rotatebox{90}{$(s_k, s_k')$}} (10_3)
  edge [loop above] node [above] {$(x,\epsilon)$} ()
 (10_4) edge [bend left] node [above] {$(y, \epsilon)$} (10_3)
  edge [loop below] node [] {$(x,x)$} ()
 edge node [above] {$(s_k, \epsilon)$} (0)
 (10_3)  edge [loop below] node [] {$(x,x)$} ()
 edge [bend left] node [below] {$(y,y)$} (10_4);

 \path (0) edge  node [above] {$(s_j, \epsilon)$} (9_1)
 (9_1) edge [<->] node [above] {$(x, \epsilon)$} (9_2)
 edge node [] {\rotatebox{90}{$(s_j, s_j'')$}} (9_4)
  edge [loop above] node [above] {$(y,\epsilon)$} ()
 (9_2) edge [<->] node [] {\rotatebox{90}{$(s_j, s_j')$}} (9_3)
  edge [loop above] node [above] {$(y,\epsilon)$} ()
 (9_4) edge [bend left] node [below] {$(x, \epsilon)$} (9_3)
  edge [loop below] node [] {$(y, y)$} ()
 edge node [above] {$(s_j, \epsilon)$} (0)
 (9_3)  edge [loop below] node [] {$(y, y)$} ()
 edge [bend left] node [above] {$(x, x)$} (9_4);
  \end{dualmoore}
  \]
  \caption{The dual Moore diagram of $\Phi_\mm$, used in the proof of
    Theorem~\ref{thm:torsion_orbit}}
  \label{fig:torsion_orbitmoore}
\end{figure}

We first note that $x$ and $y$ commute in $G$.  If $g = s_i x^my^n$ and
$s_i$ is an instruction of type III, then the orbit of $0$ under the
action of $g$ is $(0, \state31, \state 32)$. There is an edge labeled
$3$ from $g$ to $s_i' x^my^n\epsilon x^m y^n = s_i' x^{2m}y^{2n}$.

Consider next $s_j$ an instruction of type IX. There are two cases.
if $g = s_j x y^n$ then the orbit of $0$ is $(0, \state92, \state94)$ and
the output is $\epsilon s_j' x y^n\epsilon$; if $g = s_j x^{2m}y^n$ then
the orbit of $0$ is $(0, \state91, \state94)$ and the output is
$\epsilon s_j''x^my^n\epsilon$.

This means that if $m=0$ then there is an edge labeled $3$ from
$s_j x^{2^m}y^{2^n}$ to $s_j' x^{2^m}y^{2^n}$, and if $m>0$ then there
is an edge labeled $3$ from $s_j x^{2^m}y^{2^n}$ to
$s_j''x^{2^{m-1}}y^{2^n}$.

The same naturally applies to instructions of type IX. Finally, for
all $m,n\in\N$ the element $s_\dagger x^my^n$ fixes $0$, and there is
an edge labeled $1$ from $s_\dagger x^my^n$ to $1$.

\subsection{Contracting automata: proof of Theorem~\ref{thm:contracting}}\label{ss:contract}
We finally explain how to make the transducers $\Phi_\mm$ of the
previous subsections contracting. We expand the definition
from the introduction:
\begin{defn}[\cite{nekrashevych:ssg}*{Definition~2.11.1}]\label{defn:contracting}
  Let $G=\langle\Phi\rangle$ be a self-automata group with
  $\Phi\colon A\times S\to S\times A$ and
  $\overline\Phi\colon A\times G\to G\times A$. For $g\in G$ and
  $u\in A^*$, the \emph{state} $g@u$ is the unique element of $G$ such
  that $(u v)^g=u^g v^{g@u}$; namely, the action of $g$ on tails of
  sequences starting with $u$.

  The group $G$ is \emph{contracting} if there exists a finite subset
  $N\subseteq G$ such that, for all $g\in G$, there exists $n(g)\in\N$
  such that $g@u\in N$ whenever $|u|\ge n(g)$.
\end{defn}

The minimal subset $N$ satisfying the definition is called the
\emph{nucleus}. In particular, one has $n@a\in N$ for all
$(a,n)\in A\times N$, so $\Phi$ induces an automaton still written
$\Phi\colon A\times N\to N\times A$. Up to replacing $S$ by
$\widetilde S\coloneqq S\cup N$ and $A$ by $\widetilde A\coloneqq A^n$
for $n$ larger than $\max_{g\in \widetilde S^2}n(g)$, thus making the
transducer process $n$ letters at a time, one may also assume
\[\overline\Phi(\widetilde A\times \widetilde S^2)\subseteq \widetilde
  S\times\widetilde A.
\]
A transducer $\Phi$ with this extra property is called \emph{nuclear}.

Note that is is probably undecidable whether a self-similar group
$\langle\Phi\rangle$ is contracting; but it is easy to decide whether
a transducer $\Phi\colon A\times S\to S\times A$ is nuclear: by
minimizing the composite transducer
$A\times S^3\to S\times A\times S^2\to S^2\times A\times S\to
S^3\times A$, find the set $\mathcal R$ of all words
$s_1s_2s_3\in S^3$ that equal $1$ in $G$. Then $\Phi$ is nuclear if
and only if for all $a\in A,s_1,s_2\in S$ there exists $s_3\in S$ such
that if $\Phi(a,s_1)=(s'_1,b)$ and $\Phi(b,s_2)=(s'_2,c)$ then
$s'_1 s'_2 s_3^{-1}\in\mathcal R$.  The more precise form of
Theorem~\ref{thm:contracting} is:
\begin{thm}\label{thm:nuclear}
  There is no algorithm that, given a nuclear transducer
  $\Phi\colon A\times S\to S\times A$ and $a\in A$ and $s\in S$,
  determines the cardinality of the orbit of $a^\infty$ under
  $\langle s\rangle$.

  There is no algorithm that, given a nuclear transducer
  $\Phi\colon A\times S\to S\times A$ and $s\in S$, determines the
  order of $s$ in $\langle\Phi\rangle$.
\end{thm}

Note that the group is not changed by these operations of replacing
$S$ by $N$ and $A$ by $A^n$. If $\Phi$ is nuclear, then
$\langle\Phi\rangle$ is contracting in the sense of the introduction,
since $|g'|\le(|g|+1)/2$ in the word metric defined by
$N$. Conversely, if $|g'|\le\lambda|g|+C$ then one may take
$N=\{g\in G\mid C/(1-\lambda)\ge|g|\}$ to see that $G$ is contracting
in the sense of Definition~\ref{defn:contracting}.

\begin{lem}\label{lem:pathcontracting}
  Let $\Phi\colon A\times S\to S\times A$ be a transducer. If there is
  a constant $N\in\N$ such that every reduced path of length $\ge N$
  in the dual Moore diagram of $\Phi$ contains an $\epsilon$ letter
  along its output, then $\langle\Phi\rangle$ is contracting.
\end{lem}
\begin{proof}
  Consider $g\in\langle\Phi\rangle$, and represent it by a word
  $w\in S^*$ of length $\ell=|g|$. Factor $w=w_1\dots w_t$ with
  $|w_i|=N$ for all $i=1,\dots,t-1$ and $|w_t|<N$.

  Then every $g'$ as in~\eqref{eq:action} is computed by following, in
  the dual Moore diagram, the path starting at $a_1$ with label $w$ on
  its input. The output label along that path is $g'$, and by
  hypothesis each time a segment $w_i$ is read, for $i<t$, an
  $\epsilon$ letter is produced for $g'$; so $|g'|\le\ell-t+1$. Now
  $t=\lceil\ell/N\rceil$, so
  \[|g'|\le\ell-\lceil\ell/N\rceil+1\le(1-1/N)|g|+1.\qedhere\]
\end{proof}

We shall modify the transducers $\Phi_\mm$ by composing them with
appropriate machines. We recall the general definition: let
$\Phi\colon A\times S\to S\times A$ and
$\Psi \colon B\times S\to S\times B$ be transducers with same stateset
$S$. Their \emph{composition} is the transducer $\Phi\circ\Psi$ with
alphabet $A\times B$, given by
\[\Phi\circ\Psi\colon (A\times B)\times S=A\times(B\times S)\overset{A\times\Psi}\longrightarrow A\times(S\times B)=(A\times S)\times B\overset{\Phi\times B}\longrightarrow(S\times A)\times B=S\times(A\times B).\]

We are given a transducer $\Phi$ with stateset
$S=\{s_1,\dots,s_\ell,x,y\}$ and alphabet $A$. We write
$G=\langle\Phi\rangle$, and freely identify words in $S^*$ with their
value in $G$. We require that $x,y$ commute.

For every $i\in\{1,\dots,\ell\}$, consider the transducer $\Phi_i$ with
alphabet $A_i = \{0,1\}$ and transitions
$\Phi_i(a,q)=(a=0 \;?\; q : \epsilon,q=s_i \;?\; 1-a : a)$.

Note (by drawing the dual Moore diagram and deleting the transitions
with $\epsilon$ output) that the only paths with input and output of
same length are of the form $s_i^{-a} w s_i^b$ for some
$a,b\in\{0,1\}$ and $w$ a word not involving $s_i$. 

Note also that for a word $w$ of form $s_j x^my^n$ 
\begin{enumerate}
\item if $i = j$ then $\Phi_i(0, w^2) = (w, 0)$ and
  $\Phi_i(1, w^2) = (w', 1)$ with $w'$ conjugate to $w$;
\item if $i \neq j$ then $\Phi_i(0, w) = (w, 0)$ and
  $\Phi_i(1, w) = (\epsilon, 1)$.
\end{enumerate}

Consider also a transducer $\Phi_0$ with alphabet $A_0=\{0,1\}^3$ and
transitions
\begin{align*}
  \Phi_0((a,b,c),s_i)&=(c=0\; ?\; s_i : \epsilon,(a,b,1-c)) \text{ for all $i$};\\
  \Phi_0((a,b,c),x)&=(a=0\; ?\; x\; : \epsilon,(1-a,b,c));\\
  \Phi_0((a,b,c),y)&=(b=0\; ?\; y : \epsilon,(a,1-b,c)).
\end{align*}
Note that, in the dual Moore diagram of $\Phi_0$, all paths with input
label of the form $s_i^{-a} x^m y^n s_j^b$ have shorter output label as
soon as $|m|+|n|\ge3$. Note also that if $w$ is a word of the form
$s_i x^m y^n$ then for all $(a,b,c)\in A_0$ we have
$\Phi_0((a,b,c),w^2)=(w',(a,b,c))$ for some permutation $w'$ of $w$;
so in particular $w'$ is conjugate to $w$. Furthermore, for
$(a,b,c)=(0,0,0)$ we get $w'=w$ in $G$.

\begin{prop}\label{prop:contracting}
  Under the hypotheses above, the transducer
  $\Phi'\coloneqq\Phi\circ\Phi_0\circ\Phi_1\circ\cdots\circ\Phi_\ell$
  generates a contracting group, and whenever we have
  $\Phi(a,(s_i x^m y^n)^t)=(s'_i x^{m'} y^{n'},a)$ in the original
  transducer we have for all $j\in\{0,1\}^{\ell+3}$ the relation
  $\Phi'((j,a),(s_i x^m y^n)^{4t})=(w ,(j,a))$, with $w$ either equal
  to $1$ or conjugate to $s'_i x^{m'} y^{n'}$. Furthermore, if
  $j = 0^{\ell + 3}$ then $w = s'_i x^{m'} y^{n'}.$
\end{prop}
\begin{proof}
  After applying the transducers $\Phi_1,\dots,\Phi_\ell$, the only
  words that don't get shortened are of the form
  $s_i^{-a} w(x,y) s_j^b$ for some $i,j\in\{1,\dots,\ell\}$ and some
  $a,b \in \{0,1\}$. These get shortened by $\Phi_0$ as soon as
  $|w|\ge3$, using the fact that $x$ and $y$ commute. It follows that
  $\langle\Phi'\rangle$ is contracting.

  Consider the transitions of $(s_i x^m y^n)^4$ in transducer
  $\Phi_1\circ\cdots\circ\Phi_\ell$. On input letter $0^\ell$ it
  produces $(s_i x^m y^n)^2$, on input letter $0\cdots1\cdots0$ with
  the `$1$' in position $i$ it produces a conjugate of
  $(s_i x^m y^n)^2$ and on all other input letters it produces
  $\epsilon$. Feed then $(s_i x^m y^n)^2$ to transducer $\Phi_0$; on
  input letter $000$ it produces $s_i x^m y^n$ and on all other input
  letters it produces a conjugate of $s_i x^m y^n$. Feed finally
  $s_i x^m y^n$ to $\Phi$ to conclude the proof.
\end{proof}

We are ready to finish the proof of Theorem~\ref{thm:nuclear}.  We
constructed an integer-labeled graph for a transducer $\Phi$, whose
vertices are elements of $G$ for Theorem~\ref{thm:torsion_orbit} or
symmetric conjugacy classes for Theorem~\ref{thm:torsion}.

By Proposition~\ref{prop:contracting}, the transducer $\Phi'$ is
contracting. Let us check that the order problems for
$\langle\Phi\rangle$ and for $\langle\Phi'\rangle$ are equivalent.

A graph for $\Phi'$ will have the same set of vertices as the graph
for $\Phi$, and Proposition~\ref{prop:contracting} shows that this new
graph has the same set of outgoing edges for each element of form
$s_i x^m y^n$, with labels multiplied by $4$ and, possibly, some new
edges to $1$ (or to $C(1)$).  Since in the old graph there were no
loops at non-identity elements, $s_* x y$ has infinite order in
$\langle\Phi'\rangle$ if and only if it has infinite order in
$\langle\Phi\rangle$, and the orbit of $(0,0^{\ell+3})^\infty$ is
infinite under the action of $s_* x y\in\langle\Phi'\rangle$ if and
only if the orbit of $0^\infty$ is infinite under the action of
$s_* x y\in\langle\Phi\rangle$.

Finally, by replacing the stateset $S$ by $\widetilde S=S\cup N$ and
$A$ by $\widetilde A=A^n$, we may assume that $\Phi'$ is nuclear.

\section{Outlook}
We proved in this article the undecidability of the order problem for
automata groups, namely groups of transformations generated by a
transducer.

If the transducer belongs to a restricted class, it may well be that
the order problem becomes decidable. Of particular importance are:
\begin{list}{}{\leftmargin=0mm\itemsep=1ex}
\item\textbf{Transducers of polynomial growth.} In a transducer $\Phi$
  (represented by a graph as in Figure~\ref{fig:grigorchuk}), let
  $\alpha(n)$ denote the number of paths of length $n$ that end in a
  non-identity state. If $\alpha(n)$ is a bounded function (as is the
  case for the Grigorchuk group), then the order problem is solvable
  in $\langle\Phi\rangle$, see~\cite{bondarenko-b-s-z:conjugacy}. What
  happens if $\alpha(n)$ is bounded by a linear function? or by a
  polynomial of degree $d$? The groups generated by such transducers
  have been considered by Sidki~\cite{sidki:acyclicity}.
\item\textbf{Reset transducers.} These are transducers $\Phi$ with
  $\Phi(a,s)=(\phi(a),\psi(a,s))$ for some functions $\phi,\psi$;
  namely, the state reached by the transducer is independent of the
  original state. These transducers are intimately connected to
  tilings, by Kari's construction~\cite{kari:nilpotency}. Gillibert
  proved in~\cite{gillibert:finiteness} that the order problem is
  unsolvable for semigroups of reset automata. Is it solvable in
  groups of reset automata?
\item\textbf{Reversible transducers.} These are transducers whose dual is
  invertible; they should be related to reversible Turing or Minsky
  machines. Is the order problem solvable for groups generated by
  reversible automata?
\item\textbf{Bireversible transducers.} These are transducers $\Phi$
  such that all $8$ transducers obtained from $\Phi$ by inverting or
  permuting the stateset and alphabet remain transducers; they give
  another point of view on square complexes (by tiling the plane with
  squares whose labels are $(a,s,a',s')$ when $\Phi(a,s)=(s',a')$).
\end{list}

We expect it to be undecidable whether a functionally recursive group
is actually an automata group (for a larger generating set), whether
an automata group is contracting, and even whether a contracting group
is finite. Again, the related questions for semigroups are known to be
undecidable by constructions in or similar
to~\cite{gillibert:finiteness}.

\end{document}